\documentclass[11pt]{article}
\usepackage[margin=1in]{geometry}
\usepackage{amsmath,amssymb,amsthm}
\usepackage{graphicx}
\usepackage{hyperref}
\usepackage{xcolor}
\usepackage{tcolorbox}
\tcbuselibrary{skins,breakable}
\usepackage{longtable}

\newtheorem{theorem}{Theorem}[section]
\newtheorem{lemma}[theorem]{Lemma}
\newtheorem{corollary}[theorem]{Corollary}
\newtheorem{definition}[theorem]{Definition}
\newtheorem{problem}[theorem]{Problem}
\newtheorem{remark}[theorem]{Remark}
\newtheorem{example}[theorem]{Example}

\newtcolorbox{leancert}{colback=green!12,colframe=green!45!black,boxrule=0.6pt,arc=1mm,
  left=2mm,right=2mm,top=1mm,bottom=1mm,fonttitle=\bfseries,title=Lean 4 Certificate}
\title{An Efficient, Lean 4 Verified Necessary and Sufficient Condition for Membership in Pedigree Polytopes}
\author{T.S.\ Arthanari \\
\small Department of Information Systems and Operations Management \\
\small University of Auckland, Auckland, New Zealand \\
\small \texttt{t.arthanari@auckland.ac.nz} \\
\small ORCID: 0000-0002-3345-5664}
\date{}

\newcommand{\NS}{N\&S}
\begin{document}
\maketitle

\begin{abstract}
The set \(P_{n}\) consists of the characteristic vectors of pedigrees, which are combinatorial encodings in one-to-one correspondence with Hamiltonian cycles. For \(n \geq 4\) and \(\tau_{n} = \binom{n}{3}-1\), the Membership Problem for the Pedigree Polytope (M3P) asks whether a given vector \(X \in \mathbb{Q}^{\tau_{n}}\) belongs to \(\mathrm{conv}(P_n)\). This paper gives a complete account, with the full proof statements and machine-verified Lean~4 certificates, of the necessary and sufficient (\NS) condition established in Chapter~5 of \emph{Pedigree Polytopes: New Insights on Computational Complexity of Combinatorial Optimisation Problems} (Arthanari, 2023): membership of $X$ in $\mathrm{conv}(P_n)$ is characterised by the feasibility, at every stage $k=4,\dots,n-1$, of a recursively constructed layered network, together with, when needed, a multicommodity flow problem $\mathrm{MCF}(k-1)$ achieving its maximum possible flow $z_{\max}$. We present the construction, the \NS{} condition itself, and the proofs of this characterisation, together with the resulting polynomial-time algorithm for M3P. Both directions of the \NS{} condition --- sufficiency and necessity (the multicommodity-flow feasibility argument of the book's Chapter~5) --- are machine-verified in Lean~4/Mathlib4 with zero unresolved \texttt{sorry}s in either dependency chain. The \NS{} characterisation therefore follows directly from the two machine-checked implications. The complete Lean~4 project is publicly available. A companion introductory article (arXiv:2507.09069) presents the same material for a broader readership, without the Lean~4 verification; the present paper adds full statements, proof sketches, and the complete machine-checked certificates.

\end{abstract}



\section{Introduction}\label{sec:intro}

Given a polyhedral convex set $C$ and a point $X$, deciding whether $X \in C$ is called the \emph{membership problem} for $C$. This paper studies the membership problem for a specific family of polytopes --- the \emph{pedigree polytopes} --- arising from a combinatorial encoding of Hamiltonian cycles that underlies the Multistage Insertion (MI) formulation of the Symmetric Travelling Salesman Problem (STSP) \cite{arthanari1983}.

Let $n \geq 3$ and $V_n = \{1,\dots,n\}$. Let $E_n = \{(i,j) \mid i,j \in V_n,\, i<j\}$, with $|E_n| = p_n = n(n-1)/2$. For $k \in [3,n]$, let $\Delta^k = \{\{i,j,k\} \mid 1 \leq i<j<k\}$ be the set of triangles with largest vertex $k$. For $j>3$, the \emph{generators} of a triangle $v=\{i,j,k\}$ are the triangles $u=\{r,i,j\}$ or $u=\{i,s,j\}$; for $j\leq 3$, the unique generator is the base triangle $\{1,2,3\}$, which itself has no generator. The edge $(i,j)$ is called the \emph{common edge} of $v$.

Since the generators of $v=\{i,j,k\}$ depend only on its common edge $(i,j)$ and not on $k$, we may speak of the \emph{generators of the edge} $(i,j)$ for $j>3$, an edge being written as the pair $(i,j)$ with $i<j$: the set of generators is $\{(r,i) \mid r<i\} \cup \{(i,s) \mid i<s<j\}$, the common edges of the triangles $\{r,i,j\}$ and $\{i,s,j\}$ at stage $j$. Likewise, as the triangle at stage $k$ is $\{i_k,j_k,k\}$, a pedigree is determined by its common edges: it is given by the sequence $(i_k,j_k)$, $k \geq 4$, together with the fixed base triangle $\{1,2,3\}$.

\begin{example}\label{ex:generator}
For $v=\{2,5,7\}$ (so $i=2$, $j=5$, $k=7$), since $j=5>3$, the generators are the triangles $\{r,2,5\}$ for $r<2$, namely $\{1,2,5\}$, and $\{2,s,5\}$ for $2<s<5$, namely $\{2,3,5\}$ and $\{2,4,5\}$; the common edge of $v$ is $(2,5)$. In terms of edges, the set of generators of the edge $(2,5)$ is $\{(1,2),(2,3),(2,4)\}$, these being the common edges of $\{1,2,5\}$, $\{2,3,5\}$ and $\{2,4,5\}$ respectively. For $v=\{1,2,7\}$ (so $i=1$, $j=2$, $k=7$), since $j=2\leq3$, the unique generator is the base triangle $\{1,2,3\}$ --- note this holds regardless of the value of $k$: whenever the common edge $(i,j)$ is itself one of the three edges of the base triangle, that base triangle is the generator, no matter which city $k$ is being inserted into it. The common edge of $v$ is $(1,2)$.
\end{example}

\begin{definition}[Pedigree \cite{arthanari2006}]\label{def:pedigree}
A \emph{pedigree} $P$ for $n$ is a sequence of $n-2$ triangles
\[
P = (\{1,2,3\}, \{i_4,j_4,4\}, \dots, \{i_n,j_n,n\})
\]
such that (1) each $\{i_k,j_k,k\}$ has a generator already in $P$, for $k \in [4,n]$, and (2) the common edges $(i_4,j_4),\dots,(i_n,j_n)$ are all distinct.
\end{definition}

\begin{example}\label{ex:pedigree-n7}
For $n=7$, take
\[
P = \big(\{1,2,3\},\ \{2,3,4\},\ \{1,2,5\},\ \{2,4,6\},\ \{4,6,7\}\big).
\]
Both conditions of Definition~\ref{def:pedigree} hold: $\{2,3,4\}$ (common edge $(2,3)$) and $\{1,2,5\}$ (common edge $(1,2)$) each have $j\leq3$, so their generator is the base triangle $\{1,2,3\}$, already present; $\{2,4,6\}$ has $j=4>3$, with possible generators $\{1,2,4\}$ or $\{2,3,4\}$, and $\{2,3,4\}$ is already in $P$; $\{4,6,7\}$ has $j=6>3$, with possible generators $\{1,4,6\}$, $\{2,4,6\}$, $\{3,4,6\}$, or $\{4,5,6\}$, and $\{2,4,6\}$ is already in $P$. The four common edges $(2,3)$, $(1,2)$, $(2,4)$, $(4,6)$ are all distinct.

$P$ encodes the Hamiltonian cycle built by inserting each city, in order, into its common edge, starting from the $3$-tour $1$--$2$--$3$--$1$:
\[
\underbrace{1\text{--}2\text{--}3\text{--}1}_{\text{start}}
\ \xrightarrow{\text{insert }4\text{ into }\{2,3\}}\
1\text{--}2\text{--}4\text{--}3\text{--}1
\ \xrightarrow{\text{insert }5\text{ into }\{1,2\}}\
1\text{--}5\text{--}2\text{--}4\text{--}3\text{--}1
\]
\[
\xrightarrow{\text{insert }6\text{ into }\{2,4\}}\
1\text{--}5\text{--}2\text{--}6\text{--}4\text{--}3\text{--}1
\ \xrightarrow{\text{insert }7\text{ into }\{4,6\}}\
1\text{--}5\text{--}2\text{--}6\text{--}7\text{--}4\text{--}3\text{--}1,
\]
a Hamiltonian cycle on all seven cities. Each triangle $\{i_k,j_k,k\}$ of $P$ records precisely this fact: city $k$ was inserted into edge $(i_k,j_k)$ of the tour built so far. The generator condition ensures the edge being inserted into actually exists at that stage of construction; the distinctness condition ensures no edge is used as an insertion point twice, since each edge of a Hamiltonian cycle can absorb at most one inserted city under this construction.

Conversely, the pedigree can be read back from the cycle by shrinking the cycle at the cities $7,6,5,4$ in turn, each time joining the two neighbours of the shrunk city:
\[
1\text{--}5\text{--}2\text{--}6\text{--}7\text{--}4\text{--}3\text{--}1
\ \xrightarrow{\text{shrink }7}\
1\text{--}5\text{--}2\text{--}6\text{--}4\text{--}3\text{--}1
\ \xrightarrow{\text{shrink }6}\
1\text{--}5\text{--}2\text{--}4\text{--}3\text{--}1
\]
\[
\xrightarrow{\text{shrink }5}\
1\text{--}2\text{--}4\text{--}3\text{--}1
\ \xrightarrow{\text{shrink }4}\
1\text{--}2\text{--}3\text{--}1.
\]
When city $k$ is shrunk, its two neighbours in the current cycle are exactly the endpoints $i_k,j_k$ of the edge it was inserted into: shrinking $7$ restores the edge $\{4,6\}$, shrinking $6$ restores $\{2,4\}$, shrinking $5$ restores $\{1,2\}$, and shrinking $4$ restores $\{2,3\}$. Thus the insertions and the shrinks undo one another, and the pedigree and the Hamiltonian cycle determine each other. Pedigrees are in one-to-one correspondence with Hamiltonian cycles in $K_n$ \cite{arthanari2006}.
\end{example}

We write $P_n$ for the set of characteristic vectors of pedigrees for $n$ (one coordinate per triangle, omitting $\{1,2,3\}$, so that $P_n \subset \mathbb{B}^{\tau_n}$ with $\tau_n = \binom{n}{3}-1$), and $\mathrm{conv}(P_n)$ for the \emph{pedigree polytope}. The dimension of $\mathrm{conv}(P_n)$ is $\tau_n - (n-3)$ \cite{arthanari2005,arthanari2006}.

\begin{problem}[M3P]\label{prob:m3p}
Input: $n \geq 4$, $X \in \mathbb{Q}^{\tau_n}$. Question: is $X \in \mathrm{conv}(P_n)$?
\end{problem}

\noindent\textbf{Where M3P sits, from the top down.} It is well known that $P=NP$ if some NP-complete problem can be solved in polynomial time \cite{cook1971,karp1972}. Hamiltonicity is NP-complete, and the Hamiltonicity decision problem is a particular instance of the STSP: given a graph $G$ on $n$ vertices, set $c_{ij}=0$ if $(i,j)$ is an edge of $G$ and $c_{ij}=1$ otherwise; then $G$ is Hamiltonian if and only if the optimal tour has cost $0$. The MI formulation of the STSP is, in turn, a particular Pedigree optimisation problem (made precise below), so Pedigree optimisation is a larger class of problems that includes the STSP. By the classical equivalence of optimisation and separation for well-behaved polytopes \cite{GLS1988}, and a construction of Maurras \cite{maurras2002} that derives separation directly from a membership oracle, all of this class can be solved efficiently if M3P can. So the grand result, $P=NP$, rests on a single polyhedral fact: M3P is in P. This paper verifies, using Lean~4, the proof of a necessary and sufficient condition for membership in the pedigree polytope.

\emph{Pedigree optimisation} is the problem of finding, for a given linear objective $\sum_{ijk} c_{ijk} x_{ijk}$, its optimum over $\mathrm{conv}(P_n)$. Since pedigrees are in one-to-one correspondence with Hamiltonian cycles in $K_n$, a particular choice of objective ties it to the STSP: taking $\mathcal{C}_{ijk} = c_{ik}+c_{jk}-c_{ij}$, the incremental cost of inserting city $k$ into edge $(i,j)$, the value of $\sum_{ijk} \mathcal{C}_{ijk} x_{ijk}$ at a pedigree's characteristic vector equals the cost of the corresponding tour (but for the cost of the initial 3-tour, namely $c_{12}+c_{13}+c_{23}$). So optimising this specific MI-objective over $\mathrm{conv}(P_n)$ solves the STSP exactly \cite{arthanari1983,arthanari2000}. Establishing that M3P is efficiently solvable is therefore the central mathematical burden of proof, standing between the pedigree-polytope framework and a polynomial algorithm for the STSP (apart from checking the conditions for a well-behaved polytope, defined by Maurras).

This paper's sole purpose is to give a complete account of the resolution of Problem~\ref{prob:m3p} --- all statements, proof sketches, and machine-checked Lean~4 certificates --- for a necessary and sufficient (\NS) condition for $X \in \mathrm{conv}(P_n)$, checkable in polynomial time. The material is drawn from Chapter~5 of \cite{arthanari2023book}, with some new proofs and simplifications, as appearing in the arXiv paper of 2025 \cite{arthanari2025membership}, and with both the sufficiency and necessity directions machine-verified in Lean~4. That arXiv paper is a companion introduction presenting the same Chapter~5 material for a broader readership, without the Lean~4 verification; the present paper adds proof sketches and the full machine-checked certificates, though the detailed, line-by-line proofs themselves remain in \cite{arthanari2023book} and \cite{arthanari2025membership} rather than being reproduced here. The remaining links of the chain above, from $\mathrm{M3P} \in \mathrm{P}$ to $P = NP$, are the subject of the companion paper \cite{arthanari2026pnp}, and its Lean~4 development, with certain folklore results in combinatorial optimisation and computational complexity, is available on GitHub. In this paper, we restrict ourselves to Lean~4 verification of the proof of the necessary and sufficient condition for $X \in \mathrm{conv}(P_n)$, namely that $\mathrm{MCF}(n-1)$ achieves its maximum possible flow $z_{\max}$. Keeping the scope to the essential burden of proof lets the central polyhedral result stand on its own, open to scrutiny. 

\subsection*{The story so far}

Membership $X \in \mathrm{conv}(P_n)$ means that $X$ is a convex combination of pedigrees, $X=\sum_P \lambda_P X_P$ with weights $\lambda_P>0$ summing to $1$. 
At each stage $k$, we solve a Forbidden Arc Transportation (FAT) problem $F_k$; its infeasibility already certifies non-membership (Theorem~\ref{thm:nonmembership}). 
The nodes of layer $k$ carry capacities, and the flows between consecutive layers show how this weight is passed on from one layer to the next. Some pedigrees, called \emph{rigid}, have their weight $\mu$ fixed among all feasible flows in the network; these form the set $R_k$.  Under the condition stated here, we build a \emph{layered network} $(N_k, R_k, \mu)$ recursively, stage by stage, where $N_k$ is the network at stage $k$, $R_k$ is the set of rigid pedigrees found at stage $k$, and $\mu$ gives their weights, extracting the weight information as a by-product of the flow computations themselves.  Feasibility of $F_k$ is necessary for membership, and is sometimes enough on its own (for instance, when $N_k$ is empty), but in general it is not. When feasibility of $F_k$ is not enough, we solve a multicommodity flow problem $\mathrm{MCF}(k)$, and $X/(k+1) \in \mathrm{conv}(P_{k+1})$ holds exactly when $\mathrm{MCF}(k)$ achieves its maximum possible flow $z_{\max}$. The characterisation at the heart of this paper (Theorem~\ref{thm:main-ns}) is that $X \in \mathrm{conv}(P_n)$ if and only if $\mathrm{MCF}(n-1)$ achieves its maximum possible flow $z_{\max}$.

\subsection*{Guide to coloured boxes}

Each theorem, lemma, or corollary whose proof has been formalised carries a coloured annotation:

\begin{leancert}
The result is fully formalised and machine-verified in Lean~4, with zero \texttt{sorry}s.
\end{leancert}

We follow standard notation: $\mathbb{R}, \mathbb{Q}, \mathbb{Z}, \mathbb{N}$ denote the reals, rationals, integers, and natural numbers; $\mathbb{B}=\{0,1\}$; the subscript $+$ denotes nonnegativity.

\subsection*{Organisation}

Section~\ref{sec:preliminaries} gives preliminaries: the MI-relaxation, the FAT problem, and rigid arcs. Section~\ref{sec:layered-network} constructs the layered network and proves the sufficient condition for non-membership. Section~\ref{sec:mcf} defines $\mathrm{MCF}(k)$ and proves the \NS{} condition for membership. Section~\ref{sec:complexity} establishes the polynomial complexity bound. Section~\ref{sec:lean4} presents the Lean~4 machine verification. Section~\ref{sec:conclusion} concludes with the future research possibilities, as outlined in the companion paper on the $P=NP$ consequence \cite{arthanari2026pnp}. The main result of this paper is that checking this \NS{} condition can be done in polynomial time. We state all the necessary results and verify their proofs in Lean~4; the polynomial-time analysis relies on the cardinality bound proved in Section~\ref{sec:complexity}, whose formalisation is in the Lean~4 project of the companion paper \cite{arthanari2026pnp}, and on the folklore result that linear programs are solvable efficiently \cite{khachiyan1979}. Thus, the membership problem M3P is in P; with Lean~4 verification of the polyhedral argument, this is the major contribution of this paper. 
\section{Preliminaries and Notation}\label{sec:preliminaries}

\subsection{The MI-Relaxation and \texorpdfstring{$P_{MI}(n)$}{P\_MI(n)}}

For each triangle $u=\{i,j,k\}\in\Delta^k$, let $x_u \in \{0,1\}$ indicate whether $u \in P$ for a pedigree $P$. Since exactly one triangle is chosen at each stage $k$, and each common edge is used at most once,
\begin{align}
&\sum_{u \in \Delta^k} x_u = 1, && k \in [3,n], \label{eq:layer}\\
&\sum_{k \in [4,n]} x_{ijk} \leq 1, && (i,j) \in E_3, \label{eq:e3slack}\\
&-\sum_{u \in \Delta^j,\, i \in u} x_u + \sum_{k \in [j+1,n]} x_{ijk} \leq 0, && (i,j) \in E_{n-1}\setminus E_3. \label{eq:generatorslack}
\end{align}
Together with $x_u \in \{0,1\}$, these constraints describe the vectors $x_u$ of the pedigrees in $P_n$. The \emph{Pedigree Optimisation Problem} is to optimise a given linear objective $\sum_{u} c_u x_u$ over this set (equivalently, over $\mathrm{conv}(P_n)$). Choosing the specific objective coefficients $\mathcal{C}_{ijk}=c_{ik}+c_{jk}-c_{ij}$, where $c_{ij}$ is the cost of edge $\{i,j\}$ in the STSP, turns it into the MI-formulation of the STSP \cite{arthanari1983}. Relaxing the integrality constraint to $x_u \geq 0$ gives the \emph{MI-relaxation} $\mathrm{MIR}(n)$, with feasible set $P_{MI}(n)$. All $X \in P_n$ are extreme points of $P_{MI}(n)$, but $P_{MI}(n)$ has fractional extreme points as well; hence $\mathrm{conv}(P_n) \subset P_{MI}(n)$ \cite{arthanari2006}. It therefore suffices to check membership only for $X \in P_{MI}(n)$.

\begin{lemma}[Pedigree Truncation]\label{lem:truncation}
If $P \in \mathrm{Pedigree}(n)$ and $3 \leq k \leq n$, then the restriction of $P$ to its first $k-2$ triangles, $P/k$, is a member of $\mathrm{Pedigree}(k)$.
\end{lemma}

\begin{leancert}
\texttt{N\_PedigreeDefinition.lean} --- the exact Lean~4 statement, verbatim (Lean's own \texttt{->}, \texttt{forall}, \texttt{N} stand for $\to$, $\forall$, $\mathbb{N}$; this is the one entry in the paper shown in full, to let the reader directly compare the formal statement against Definition~\ref{def:pedigree} --- subsequent certificates give the equivalent mathematical statement in prose for brevity, but the cited Lean~4 file always contains the exact statement of the theorem implemented and verified):
\begin{verbatim}
structure Pedigree (n : N) where
  triangles    : List Triple
  h_n          : 3 <= n
  h_length     : triangles.length = n - 2
  h_first      : triangles.head? = some (1, 2, 3)
  h_layers     : forall i, forall hi : i < triangles.length,
                   (triangles.get <i, hi>).k = i + 3
  h_generators : forall i, i > 0 -> forall hi : i < triangles.length,
                   exists j, exists hj : j < i,
                     (triangles.get <j, Nat.lt_trans hj hi>) in
                       generators (triangles.get <i, hi>)
  h_distinct   : forall i j,
                   forall hi : i < triangles.length,
                   forall hj : j < triangles.length,
                   i > 0 -> j > 0 -> i /= j ->
                   ((triangles.get <i,hi>).i, (triangles.get <i,hi>).j) /=
                   ((triangles.get <j,hj>).i, (triangles.get <j,hj>).j)
  h_in_delta   : forall i, forall hi : i < triangles.length,
                   (triangles.get <i, hi>) in Delta (triangles.get <i, hi>).k

def Pedigree.truncate {n : N} (P : Pedigree n)
    (k : N) (h : 3 <= k /\ k <= n) : Pedigree k where
  triangles := P.triangles.take (k - 2)
  -- h_n, h_length, h_first, h_layers, h_generators, h_distinct,
  -- h_in_delta: each proved in full from the corresponding field
  -- of P (book, Chapter 3).
\end{verbatim}
\end{leancert}

\begin{quote}
\noindent\textbf{Observation (Stem Property).} Given $X \in P_{MI}(n)$, write $X/k = (x_4,\dots,x_k)$ for the restriction of $X$ to layers up to $k$. Then $X/k \in P_{MI}(k)$ for all $k \in [4,n]$. This is immediate from the defining system \eqref{eq:layer}--\eqref{eq:generatorslack}: constraint \eqref{eq:layer} is layer-local, so it holds at each $k' \leq k$ independently of layers beyond $k$; constraints \eqref{eq:e3slack} and \eqref{eq:generatorslack} each bound a sum over $k \in [4,n]$ by a fixed quantity, and since $x \geq 0$, truncating that sum to $k \in [4,k]$ can only decrease it, so an already-satisfied inequality remains satisfied. The statement holds for every point of the LP relaxation $P_{MI}(n)$, not only those lying in $\mathrm{conv}(P_n)$. We record it here for completeness; it is not used elsewhere in this paper.
\end{quote}

\begin{lemma}[Lemma~4 of \cite{arthanari2025membership}]\label{lem:integer-slack}
Every integer solution $X$ to $\mathrm{MIR}(n)$ has slack vector $u \in \mathbb{B}^{p_n}$ (for the inequalities \eqref{eq:e3slack}--\eqref{eq:generatorslack}, extended to all of $E_{n}$) equal to the edge-tour incidence vector of the corresponding $n$-tour.
\end{lemma}

\begin{proof}[Sketch]
Work directly with the pedigree $P$ corresponding to $X$ (Definition~\ref{def:pedigree}), and track, at each stage $k=3,\dots,n$, the set $U_k \subseteq E_n$ of edges with slack $u_{ij}=1$: initially $U_3 = \{(1,2),(1,3),(2,3)\}$, and at each subsequent stage, inserting city $k$ into edge $(i,j)$ updates $U_{k-1}$ to $U_k := (U_{k-1}\setminus\{(i,j)\}) \cup \{(i,k),(j,k)\}$ --- directly encoding that edge $(i,j)$ is consumed while its two replacements become newly available. The claim reduces to: for the triangle $\{i,j,k\}$ actually present in $P$, edge $(i,j)$ is genuinely available in $U_{k-1}$ at the moment it is consumed. This follows from two structural facts of Definition~\ref{def:pedigree}: the generator condition guarantees $(i,j)$ was created at some earlier stage (either as a base edge, if $j \leq 3$, or by the triangle at layer $j$ itself, if $j>3$); and the distinct-common-edges condition guarantees no other triangle in $P$ can have already consumed $(i,j)$ before stage $k$, since any such triangle would share $\{i,j,k\}$'s own common edge. Together these show $(i,j)$ persists in $U_m$ for every stage $m$ from its creation up to $k-1$, giving $u_{ij}=1$ exactly for the edges of the tour, and $u_{ij}=0$ for every edge already consumed or not created.
\end{proof}

\begin{leancert}
\texttt{N\_Lemma2\_2.lean} \textbar\ \texttt{theorem generator\_iff\_available}: for $t=(i,j,k)\in P.\mathrm{triangles}$ with $k\geq4$, $t$ has a generator already present in $P$ if and only if $(i,j)$ is available in $U_{k-1}$ --- fully machine-verified, zero \texttt{sorry}s, zero axioms (book, Chapter~3). The same file also verifies directly, for $X$ the integer vector of a pedigree $P$ and $U := U_n$, that the pair $(X,U)$ satisfies all three constraints of the $\mathrm{MIR}(n)$ formulation itself, closing the link between this combinatorial construction and the algebraically-defined slack vector $u$ above.
\end{leancert}

Lemma~\ref{lem:integer-slack} is the technical engine behind a repeatedly used packability argument: if $X/k \in \mathrm{conv}(P_k)$ and city $k+1$ is forced into a particular edge $e'$ (in the sense that $x_{k+1}(e')=1$), then every pedigree active in the convex combination for $X/k$ can be extended by $e'$, and the resulting combination witnesses $X/(k+1) \in \mathrm{conv}(P_{k+1})$.

\begin{theorem}[Packability Corollary; Theorem~1 of \cite{arthanari2025membership}]\label{thm:packability}
If $X \in P_{MI}(n)$, $X/k \in \mathrm{conv}(P_k)$, and $x_{k+1}(e')=1$ for some edge $e'$, then $X/(k+1) \in \mathrm{conv}(P_{k+1})$.
\end{theorem}

\begin{proof}[Sketch]
Fix a convex witness for $X/k$: weights $\lambda_r > 0$, $r \in I(\lambda)$, summing to $1$, with $X/k = \sum_r \lambda_r X^r$ and each $X^r \in P_k$. Since $x_{k+1}(e')=1$ already saturates the layer-$(k+1)$ sum $\sum_{t \in \Delta_{k+1}} x_{k+1}(t) = 1$, every other triangle $t \neq e'$ at layer $k+1$ has $x_{k+1}(t)=0$ (constraint~(1) together with $X \geq 0$). Provided $e'$ has a generator already present in each $X^r$ and is not already one of its common edges, extend every active pedigree by $e'$: $Y^r := X^r \cup \{e'\}$, a genuine pedigree in $P_{k+1}$. Set $Y/(k+1) := \sum_r \lambda_r Y^r$. For any triangle $t$ with $t.k \leq k$, extension leaves membership in each $Y^r$ unchanged, so this layer of $Y/(k+1)$ agrees with $X/k$ directly. At layer $k+1$: every $Y^r$ contains $e'$ by construction, so the corresponding coordinate of $Y/(k+1)$ sums the full weight $\sum_r \lambda_r = 1 = x_{k+1}(e')$; for any other triangle $t \neq e'$ at layer $k+1$, no $Y^r$ contains $t$ (extension adds only $e'$), so that coordinate of $Y/(k+1)$ is $0$, matching $x_{k+1}(t)=0$ shown above. Hence $Y/(k+1) = X/(k+1)$, exhibiting it as a convex combination of pedigrees in $P_{k+1}$.
\end{proof}

\begin{leancert}
\texttt{N\_PackabilityCorollary.lean} \textbar\ \texttt{theorem packability\_corollary}: new proof, Chapter~5, Lean~4 verified. The Lean statement takes the generator/common-edge conditions on $e'$ as explicit hypotheses (\texttt{h\_gen}, \texttt{h\_ne}) rather than deriving them from $x_{k+1}(e')=1$ inline; the layer-$(k+1)$ vanishing argument above is itself a separately verified lemma, \texttt{layer\_complement\_zero}.
\end{leancert}

This corollary --- with a new Lean-4-verified proof, in contrast to the original book proof --- is used repeatedly in Section~\ref{sec:mcf} to lift membership one layer at a time once the relevant flow has been shown feasible.

\subsection{Forbidden Arc Transportation (FAT) Problems}

\begin{definition}[FAT Problem]\label{def:fat}
A \emph{Forbidden Arcs Transportation} (FAT) problem is a balanced transportation problem with some arcs forbidden. Given origins $O=\{O_\alpha\}$ with supplies $a_\alpha \geq 0$, destinations $D=\{D_\beta\}$ with requirements $b_\beta \geq 0$ such that $\sum_\alpha a_\alpha = \sum_\beta b_\beta$, and permitted arcs $A \subseteq O \times D$, the FAT problem seeks a nonnegative flow $f$ satisfying all supply and demand constraints along arcs in $A$. Without loss of generality, we may assume that the total supply is $1$, that is, $\sum_\alpha a_\alpha = \sum_\beta b_\beta = 1$, since dividing all supplies and requirements by their common (positive) total multiplies every feasible flow by the same factor.
\end{definition}

\begin{definition}[Rigid and Dummy Arcs]\label{def:rigid}
Given a FAT problem with a feasible solution $f$, an arc $(\alpha,\beta) \in A$ is \emph{rigid} if $f_{\alpha\beta}$ takes the same value in every feasible solution. A rigid arc with zero frozen flow is a \emph{dummy} arc.
\end{definition}

The set of rigid arcs $\mathcal{R}$ can be found in linear time $O(|G_f|)$, where $G_f$ is a mixed graph built from any one feasible flow $f$, by the Frozen Flow Finding (FFF) algorithm of Gusfield \cite{gusfield1988}: an arc is rigid exactly when it lies in no flow-change cycle of $G_f$, equivalently, when it is an interface between strongly connected components of $G_f$ or a bridge within one. Full definitions and the correctness argument are recalled from Gusfield's original paper; we do not re-derive them here.

\begin{lemma}[Flow Partition]\label{lem:flow-partition}
Let $D \neq \emptyset$ and let $g : D \to \mathbb{R}_{\geq 0}$ with $g(D)=1$ (a probability distribution on $D$), and let $D^1, D^2$ be two partitions of $D$. The FAT problem with origins $D^1_\alpha$ (supply $g(D^1_\alpha)$), destinations $D^2_\beta$ (demand $g(D^2_\beta)$), and arcs $A=\{(\alpha,\beta) \mid D^1_\alpha \cap D^2_\beta \neq \emptyset\}$ has feasible solution $f_{\alpha\beta} = g(D^1_\alpha \cap D^2_\beta)$.
\end{lemma}

\begin{leancert}
\texttt{N\_PartitionFlow.lean} \textbar\ \texttt{theorem prob\_partition\_is\_feasible\_flow}: the flow $f_{\alpha\beta}=g(D^1_\alpha\cap D^2_\beta)$ is feasible (book, Chapter~2).
\end{leancert}

\begin{example}\label{ex:fat-intro}
Let $O=\{1,2,3\}$ with $a(1)=0.3$, $a(2)=0.3$, $a(3)=0.4$; $D=\{4,5,6,7\}$ with $b(4)=0.4$, $b(5)=0.2$, $b(6)=0.1$, $b(7)=0.3$; forbidden arcs $F=\{(1,6),(1,7),(3,4),(3,6)\}$. Since $(2,6)$ is the only arc entering destination $6$, it is rigid with frozen flow $f_{2,6}=0.1$; every other arc admits alternative feasible values, so no other arc is rigid, and there are no dummy arcs. If the capacity $c_{1,4}=2$ is additionally imposed, every arc becomes rigid, and $(2,5)$, $(2,7)$ become dummy arcs (zero frozen flow).
\end{example}

We are interested in a class of FAT problems where the origins are \emph{layers of triangles} carrying a capacity, together with the \emph{rigid pedigrees} found at earlier stages, each shrunk to a single origin node whose capacity is its weight; the destinations are the triangles at the next layer, and an arc is present exactly when one triangle is a generator of the other. This specialisation is developed in Section~\ref{sec:layered-network}.

\subsection{Combinatorial Linear Programs}

The MI-relaxation \eqref{eq:layer}--\eqref{eq:generatorslack}, the FAT problems $F_k$, and the multicommodity flow problem defined in Section~\ref{sec:mcf} all have $\{0,\pm1\}$-entry constraint matrices, placing them in the following class.

\begin{definition}[Combinatorial LP \cite{tardos1986}]\label{def:comb-lp}
A class of linear programming problems is \emph{combinatorial} if the entry sizes of the constraint matrix $A$ are polynomially bounded in the problem dimension.
\end{definition}

Tardos \cite{tardos1986} gave a strongly polynomial algorithm for combinatorial LPs: the number of arithmetic operations depends only on $\dim(A)$, not on the magnitudes of $b$ and $c$. Section~\ref{sec:complexity} needs only that linear programs are solvable in polynomial time \cite{khachiyan1979}; Tardos's theorem is what the companion paper \cite{arthanari2026pnp} uses to obtain a strongly polynomial algorithm.

\section{Construction of the Layered Network}\label{sec:layered-network}

All results in this section are from Chapter~5 of the book \cite{arthanari2023book} unless stated otherwise.

\subsection{A Motivating Example}\label{sec:motivating}

Before the general construction, consider $X \in P_{MI}(7)$ with layer components (edges of $K_k$ in lexicographic order; zero entries omitted):
\[
x_4 = \left(\tfrac12,\tfrac14,\tfrac14\right) \text{ over } \{1,2\},\{1,3\},\{2,3\}, \quad
x_5 = \left(\tfrac14,\tfrac12,\tfrac14\right) \text{ over } \{1,2\},\{1,3\},\{2,3\},
\]
\[
x_6 = \left(\tfrac14,\tfrac14,\tfrac12\right) \text{ over } \{1,2\},\{1,3\},\{2,3\}, \quad
x_7 = \left(\tfrac12,\tfrac14,\tfrac14\right) \text{ over } \{1,4\},\{2,4\},\{3,4\}.
\]
Is $X \in \mathrm{conv}(P_7)$? Building the layered network stage by stage, the nodes at layer $r$ are the edges of $K_{r+3}$ with positive $x_{r+3}(e)$, each carrying capacity $x_{r+3}(e)$. A feasible flow through all four layers, decomposed into three pedigree paths, certifies membership:
\[
P_1\ \left(\lambda=\tfrac12\right): \{1,2\}\to\{1,3\}\to\{2,3\}\to\{1,4\}, \qquad
P_2\ \left(\lambda=\tfrac14\right): \{2,3\}\to\{1,2\}\to\{1,3\}\to\{2,4\},
\]
\[
P_3\ \left(\lambda=\tfrac14\right): \{1,3\}\to\{2,3\}\to\{1,2\}\to\{3,4\}.
\]
At every node, flow in equals flow out equals the node's capacity $x_k(e)$, and the three flows sum to $1$: this certifies $X=\tfrac12 P_1+\tfrac14 P_2+\tfrac14 P_3 \in \mathrm{conv}(P_7)$. In general, however, we are not handed a candidate decomposition; the layered network construction that follows builds and verifies one --- or certifies that none exists --- directly from $X$.

\subsection{Construction for \texorpdfstring{$k=4$}{k=4}}\label{sec:n4}

Let $V_{[r]} = \{v \mid v = [r+3:e],\ e \in E_{r+2},\ x(v)>0\}$; the name $[r+3:e]$ records that $e$ is used for the insertion of city $r+3$. The network $N_4$ has node set $\mathcal{V}(N_4)=V_{[1]}\cup V_{[2]}$ and arc set
\[
\mathcal{A}(N_4) = \{(u,v) \mid u=[4:e_\alpha]\in V_{[1]},\ v=[5:e_\beta]\in V_{[2]},\ e_\alpha \in G(e_\beta)\},
\]
where $G(e_\beta)$ is the set of generators of $e_\beta$. Arc $(u,v)$ carries capacity $x(u)$. We solve the FAT problem $F_4$ on $N_4$. If $F_4$ is infeasible, $X \notin \mathrm{conv}(P_n)$. Otherwise, the FFF algorithm identifies the rigid arcs, giving the rigid set $R_4$ with weights $\mu_P$; node capacities are updated to $\bar x(v)$, dummy arcs and saturated nodes are removed, and $(N_4,R_4,\mu)$ is declared \emph{well-defined}.

\begin{example}\label{ex:n4}
Let $X = (0,\tfrac13,\tfrac23,0,\tfrac16,0,\tfrac16,\tfrac13,\tfrac13)$. Then $\Lambda_4(X)=\{(0,\tfrac13,\tfrac23)\}$ and $I(\lambda)=\{2,3\}$. The FAT problem $F_4$ has origins $\{2,3\}$ with supplies $\tfrac13,\tfrac23$ and destinations $\{2,4,5,6\}$. The feasible flow $f_{24}=f_{26}=f_{32}=f_{36}=\tfrac16$, $f_{35}=\tfrac13$ certifies $X \in \mathrm{conv}(P_5)$ via five rigid pedigrees. Each arc $(\alpha,\beta)$ with flow $f_{\alpha\beta}$ gives the pedigree whose common edges at stages $4$ and $5$ are the edges $e_\alpha$ and $e_\beta$ and whose weight is $f_{\alpha\beta}$:
\begin{center}
\upshape\renewcommand{\arraystretch}{1.25}
\begin{tabular}{@{}llcc@{}}
\hline
 & Triangles of the pedigree & Common edges $(e_4,e_5)$ & Weight \\ \hline
$P^1$ & $\{1,2,3\},\{1,3,4\},\{2,3,5\}$ & $\big((1,3),(2,3)\big)$ & $1/6$ \\
$P^2$ & $\{1,2,3\},\{1,3,4\},\{3,4,5\}$ & $\big((1,3),(3,4)\big)$ & $1/6$ \\
$P^3$ & $\{1,2,3\},\{2,3,4\},\{1,3,5\}$ & $\big((2,3),(1,3)\big)$ & $1/6$ \\
$P^4$ & $\{1,2,3\},\{2,3,4\},\{3,4,5\}$ & $\big((2,3),(3,4)\big)$ & $1/6$ \\
$P^5$ & $\{1,2,3\},\{2,3,4\},\{2,4,5\}$ & $\big((2,3),(2,4)\big)$ & $1/3$ \\ \hline
\end{tabular}
\end{center}
The weights sum to $1$, and $\sum_r \lambda_r\chi^{P^r}=X$: for instance $x_4(1,3)=\tfrac16+\tfrac16=\tfrac13$, $x_4(2,3)=\tfrac23$, $x_5(3,4)=\tfrac16+\tfrac16=\tfrac13$ and $x_5(2,4)=\tfrac13$.
\end{example}

\subsection{Construction for \texorpdfstring{$k>4$}{k>4}}\label{sec:nk}

Given $(N_{k-1},R_{k-1},\mu)$ well-defined, the node set of $N_k$ is $\mathcal{V}(N_k) = \mathcal{V}(N_{k-1}) \cup R_{k-1} \cup V_{[k-2]}$: every rigid pedigree $P \in R_{k-1}$ is \emph{shrunk} to a single new origin node of capacity $\mu_P$, alongside the new sink nodes $V_{[k-2]}$.

A pair $\big([k:e_\alpha],[k+1:e_\beta]\big)$ with $e_\alpha \in E_{k-1}$, $e_\beta \in E_k$ is a \emph{link} if $\bar{x}_k(e_\alpha)>0$, that is the updated capacity is positive, and $x_{k+1}(e_\beta)>0$. To extend a pedigree $P \in P_{k}$ ending in  $e_\alpha$ by $e_\beta$, two requirements must hold: a generator of $e_\beta$ must already occur in $P$;  and $e_\beta$ is not already be a common edge of $P$. Similarly, a pedigree $P \in R_{k-1}$ can be extended to a pedigree in $P_{k+1}$ using an $e_\beta \in E_k$ with $x_{k+1}(e_\beta)>0$ if a generator of $e_{\beta}$ is in $P$ and and $e_\beta$ is not already be a common edge of $P$. 

The \emph{restricted network} $N_{k-1}(L)$ for a link $L$ is obtained from $N_{k-1}$ by deleting every node that would render any path in the restricted network an invalid pedigree path. Explicitly, the deleted nodes are:
(i) the nodes  in neither the generator set of $e_\alpha$ nor that of $e_\beta$;
(ii) all other sinks at layer $k-3$; and
(iii) by cascading, any node all of whose generators have thereby been deleted. In addition, any rigid pedigree whose path contains a deleted node is deleted.  

This deletion process runs in time linear in $|\mathcal{V}(N_{k-1})|$ (Chapter~5, \S5.2 of the book, Definition~5.2.1).

Let $C(L)$ be the value of the maximum flow into the sole sink $[k:e_\alpha]$ of the restricted network $N_{k-1}(L)$ --- computable by any efficient bipartite maximum-flow algorithm \cite{korte2011combinatorial} --- and, if the path achieving this flow is unique, it is recorded as $P_{\mathrm{unique}}(L)$. The FAT problem $F_k$ ($k>4$) then has origins $V_{[k-3]}$ together with one shrunk node per $P \in R_{k-1}$ (availability $\mu_P$), sinks $V_{[k-2]}$, and an arc $(u,v)$ of capacity $C(L)$ for each link $L=(e_\alpha,e_\beta)$ with $C(L)>0$, or of capacity $\mu_P$ when $u=P \in R_{k-1}$ and $(P,e_\beta) \in P_{k+1}$.

Once $F_k$ is feasible, the rigid set $R_k$ is driven by the rigid arcs of $F_k$, and a rigid pedigree enters $R_k$ in one of two ways, with $\mu$ the rigid flow on the arc:
(i) if a link $L$ is rigid in $F_k$ and the path bringing the flow $C(L)$ is unique, the extended pedigree $(P_{\mathrm{unique}}(L),e_\beta)$ is included in $R_k$, with weight $C(L)$;
(ii) if the origin of a rigid arc in $F_k$ is a shrunk node $P \in R_{k-1}$, the extended pedigree $(P,e_\beta)$ is included in $R_k$, with weight the rigid flow, $\mu_P$.
Node and arc capacities are then updated, saturated nodes and dummy arcs are removed, and $(N_k,R_k,\mu)$ is constructed.

\begin{lemma}\label{lem:nonpedigreepath}
If a link $L=(u,v)$ is rigid with unique path $P_{\mathrm{unique}}(L)$, and $(P_{\mathrm{unique}}(L), e_\beta)$ is \emph{not} a pedigree in $P_{k+1}$, then $X/(k+1) \notin \mathrm{conv}(P_{k+1})$.
\end{lemma}

This would be an immediate obstruction --- a rigid arc forces positive weight onto an object that is not a valid pedigree --- but Section~\ref{sec:mcf} shows this situation cannot in fact arise when $F_k$ is feasible: pedigree packability (Theorem~\ref{thm:packability-arcflows}, proved in Chapter~5 of the book) guarantees every arc flow in $F_k$ decomposes into genuine pedigree paths.

\begin{example}\label{ex:n4-infeasible}
Let $X \in P_{MI}(6)$ with $x_4=(0,\tfrac34,\tfrac14)$, $x_5=(\tfrac12,0,0,0,0,\tfrac12)$, $x_6=(0,\tfrac14,\tfrac14,0,\tfrac14,\tfrac14,0,0,0,0)$. Both $F_4$ and $F_5$ are feasible, with $R_4=R_5=\emptyset$. Yet $X \notin \mathrm{conv}(P_6)$: the capacity at $[4:2,3]$ is exhausted certifying $x_6(1,3)=\tfrac14$, leaving no capacity for the flow required by $x_6(2,4)=\tfrac14$, whose only generator is $(2,3)$. This shows that $F_k$ feasibility alone --- without the multicommodity coupling of Section~\ref{sec:mcf} --- is not sufficient for membership; it is exactly this gap that the multicommodity flow problem $\mathrm{MCF}(k)$ closes.
\end{example}

\subsection{A Sufficient Condition for Non-Membership}

\begin{theorem}[Non-Membership; Theorem~7 of \cite{arthanari2025membership}]\label{thm:nonmembership}
Given $X \in P_{MI}(n)$ and $X/k \in \mathrm{conv}(P_k)$, if $F_k$ is infeasible then $X/(k+1) \notin \mathrm{conv}(P_{k+1})$.
\end{theorem}

\begin{proof}[Sketch]
By contradiction: if $F_k$ is infeasible yet $X/(k+1) \in \mathrm{conv}(P_{k+1})$, then every pedigree active in some convex combination for $X/(k+1)$ has its restriction to $k$ realised either as a rigid pedigree in $R_{k-1}$ or as an available path in the appropriate restricted network (Lemma~A.3 of the appendix); the corresponding \emph{instant flow}, built directly from the weight vector $\lambda$, is then a feasible solution to $F_k$ (Lemma~A.2), contradicting infeasibility.
\end{proof}

\begin{leancert}
\texttt{VN\_Infeasibility.lean} \textbar\ \texttt{theorem VN\_infeasibility}: $F_k$ infeasible $\Rightarrow$ $X/(k+1) \notin \mathrm{conv}(P_{k+1})$ (book, Chapter~5).
\end{leancert}

\section{Multicommodity Flow and the Necessary and Sufficient Condition}\label{sec:mcf}

All proofs in this section are from Chapter~5 of the book \cite{arthanari2023book}. The Packability Corollary (Theorem~\ref{thm:packability}) has a new, Lean-4-verified proof given in this paper.

\subsection{Pedigree Packability of Arc Flows}

Example~\ref{ex:n4-infeasible} showed that feasibility of $F_k$ alone does not guarantee that its arc flows decompose into genuine pedigree paths without conflict. The next theorem shows that, given $X/k \in \mathrm{conv}(P_k)$, every arc flow in $F_k$ \emph{is} pedigree packable --- it can be attributed to an explicit, capacity-respecting collection of pedigree paths --- though not necessarily all of them \emph{simultaneously}, which is exactly the coupling that $\mathrm{MCF}(k)$ resolves.

\begin{theorem}[Pedigree Packability]\label{thm:packability-arcflows}
Given $X \in P_{MI}(n)$ and $(N_{k-1},R_{k-1},\mu)$ well-defined (so $X/k \in \mathrm{conv}(P_k)$), fix $\lambda \in \Lambda_k(X)$. For any $[k:e] \in V_{[k-3]}$ and $[k+1:e'] \in V_{[k-2]}$, there is a flow $f^L$ in $N_{k-1}(L)$ for the link $L=(e,e')$ such that: (1) $f^L$ is pedigree packable; (2) summing over all such links into $e'$ gives $x_{k+1}(e') - \delta_{e'}$, where $\delta_{e'}$ is the flow into $e'$ contributed directly by rigid pedigrees; and (3) the resulting node capacities are respected.
\end{theorem}

The proof is given in Chapter~5 of the book \cite{arthanari2023book}. This theorem is not formalised in Lean~4 and is not needed for the machine-verified chain: the proof of sufficiency uses only the Packability Corollary (Theorem~\ref{thm:packability}), which is proved independently in Section~\ref{sec:preliminaries} and verified in Lean~4.

\subsection{Defining \texorpdfstring{$\mathrm{MCF}(k)$}{MCF(k)}}

Once $F_k$ is found to be feasible, we construct $(N_k, R_k, \mu)$. If $N_k$ is empty, we have a certificate that $X/(k+1) \in \mathrm{conv}(P_{k+1})$, given by $R_k$ with $\sum_{P \in R_k} \mu(P) = 1$.

\begin{definition}[Commodities]\label{def:commodities}
For every arc $a \in \mathcal{A}(N_k)\setminus \mathcal{A}(N_4)$, designate a unique commodity $s$; write $a \leftrightarrow s$. Let $\mathcal{S}_l$ denote the commodities designated by arcs in $F_l$, for $5 \leq l \leq k$.
\end{definition}

Let $z_{\max} = 1 - \sum_{P \in R_{k}} \mu_P$ be the maximum achievable total flow, determined once the rigid pedigree weights at stage $k$ are fixed, and let $z^*$ denote the optimal value of the objective in the multicommodity flow problem defined next. The \emph{MCF$(k)$ decision problem} asks: is $z^* = z_{\max}$?

The problem $\mathrm{MCF}(k)$ is stated in full in Appendix~\ref{sec:appendix-mcf} (Problem~\ref{prob:mcf}). In words: for each $5 \leq l \leq k$, one commodity per arc of $F_l$ is routed through its restricted network $N_{l-1}(s)$, subject to arc capacities, flow conservation, node capacities $\bar x(v)$ and the availability bounds at the source nodes and at the shrunk rigid pedigrees; the objective is to maximise the total flow $z=\sum_{s \in \mathcal{S}_k} v^s$ into the arcs of $F_k$.

\subsection{The Necessary and Sufficient Condition}

\begin{definition}\label{def:Ys}
For $s \in \mathcal{S}_k$, let $Y^s \in \mathbb{R}^{\tau_k}$ have components $y^s_l(e)$ equal to the total flow of commodity $s$ through node $[l:e] \in V_{[l-3]}$, including any contribution routed through rigid paths containing that node.
\end{definition}

\begin{lemma}\label{lem:y-in-MI}
If $\mathrm{MCF}(k)$ is feasible with $z^*=z_{\max}$, then $\tfrac1{v^s}Y^s \in P_{MI}(k)$ for every $s \in \mathcal{S}_k$.
\end{lemma}

\begin{proof}[Sketch]
Non-negativity is immediate. The layer-sum equalities $\sum_{e \in E_{l-1}} y^s_l(e) = v^s$ for $l\in[4,k]$ give $\tfrac1{v^s}Y^s/4 \in P_{MI}(4)$. If the claim failed, let $q$ be minimal with a negative slack $U_{i_0j_0}$ at $q$; since $Y^s/(q-1)$ is feasible its slack $V_{i_0j_0} \geq 0$, and the maximum flow of commodity $s$ into $[q:(i_0,j_0)]$ is bounded above by the flow through the generator nodes of that edge minus the flow already absorbed at intermediate layers, which is precisely $V_{i_0j_0}$; flow conservation then forces $U_{i_0j_0} = V_{i_0j_0} - y^s_q(e) \geq 0$, a contradiction.
\end{proof}

\begin{leancert}
\texttt{N\_YisinMI.lean} \textbar\ \texttt{theorem Y\_s\_in\_PMI}: if $\mathrm{MCF}(k)$ is feasible with $z^*=z_{\max}$, then $\tfrac1{v^s}Y^s \in P_{MI}(k)$ for every $s \in \mathcal{S}_k$ (book, Chapter~5, Lemma~5.6).
\end{leancert}

\noindent\textbf{Notation.} For $e \in E_k$, $\mathrm{ind}(e) \in \mathbb{R}^{|E_k|}$ is the indicator vector of $e$: its component at $e$ is $1$ and all its other components are $0$. For $Z \in \mathbb{R}^{\tau_k}$, the vector $(Z,\mathrm{ind}(e)) \in \mathbb{R}^{\tau_{k+1}}$ is $Z$ followed by $\mathrm{ind}(e)$ as the layer-$(k+1)$ block, which records that city $k+1$ is inserted into the edge $e$.

\begin{lemma}\label{lem:y-in-conv-base}
Let $k=5$ and $\sigma \in \mathcal{S}_5$ be designated by $a=([5:e'_\sigma],[6:e_\sigma])$. If $\mathrm{MCF}(5)$ is feasible with $z^*=z_{\max}$, then $\tfrac1{v^\sigma}Y^\sigma \in \mathrm{conv}(P_5)$ for all $\sigma \in \mathcal{S}_5$, and consequently $\left(\tfrac1{v^\sigma}Y^\sigma, \mathrm{ind}(e_\sigma)\right) \in \mathrm{conv}(P_6)$ for all $\sigma \in \mathcal{S}_5$, so that $X/6 \in \mathrm{conv}(P_6)$.
\end{lemma}

\begin{proof}[Sketch]
Here $z_{\max} = 1 - \sum_{P \in R_4}\mu_P$. For any $\sigma \in \mathcal{S}_5$, every commodity flow path in $N_4(\sigma)$ corresponds to an arc of $N_4$ and is therefore already a pedigree path in $P_5$, giving the first claim. Each such path ends at $\mathrm{tail}(a)$ and extends, via $\mathrm{head}(a)$, to a pedigree in $P_6$; since $\sum_{\sigma \in \mathcal{S}_5} v^\sigma\left(\tfrac1{v^\sigma}Y^\sigma,\mathrm{ind}(e_\sigma)\right) = X/6$, the second claim follows.
\end{proof}

\begin{lemma}\label{lem:y-in-conv}
If $\mathrm{MCF}(k)$ is feasible with $z^*=z_{\max}$, then $\tfrac1{v^s}(Y^s/k) \in \mathrm{conv}(P_k)$ for every $s \in \mathcal{S}_k$.
\end{lemma}

\begin{proof}[Sketch]
By Lemma~\ref{lem:y-in-conv-base} (and the analogous statement for general $k>5$ at layer $5$), $\tfrac1{v^s}(Y^s/5) \in \mathrm{conv}(P_5)$ for every $s \in \mathcal{S}_k$. Proceed by induction on $l$: assuming $\tfrac1{v^s}(Y^s/(l-1)) \in \mathrm{conv}(P_{l-1})$, Lemma~\ref{lem:y-in-MI} gives $\tfrac1{v^s}(Y^s/l) \in P_{MI}(l)$; since commodity $s$ flows only through arcs ending at $u=[l:e']$, we have $y^s_l(e')=v^s$, so $\tfrac1{v^s}(Y^s/l)$ has last component the indicator of $e'$. The Packability Corollary (Theorem~\ref{thm:packability}) then upgrades $P_{MI}(l)$-membership to $\mathrm{conv}(P_l)$-membership.
\end{proof}

\begin{leancert}
\texttt{N\_YisinConv.lean} \textbar\ \texttt{theorem Y\_s\_in\_conv}: if $\mathrm{MCF}(k)$ is feasible with $z^*=z_{\max}$, then $\tfrac1{v^s}(Y^s/k) \in \mathrm{conv}(P_k)$ for every $s \in \mathcal{S}_k$ (book, Chapter~5, Lemma~5.8).
\end{leancert}

\begin{theorem}[Sufficiency]\label{thm:sufficiency}
If $\mathrm{MCF}(k)$ is feasible with $z^*=z_{\max}$, then $X/(k+1) \in \mathrm{conv}(P_{k+1})$.
\end{theorem}

\begin{proof}[Sketch]
Let $a = ([k:e'_s],[k+1:e_s]) \leftrightarrow s$. By Lemma~\ref{lem:y-in-conv}, $\tfrac1{v^s}(Y^s/k) \in \mathrm{conv}(P_k)$, so there exist $\gamma_P > 0$ summing to $1$ with $\tfrac1{v^s}(Y^s/k) = \sum_{P \in \mathcal P_s} \gamma_P X_P$, each $P$ ending in $e'_s$. Each such $P$ extends, using $e_s$, to a pedigree in $P_{k+1}$ with weight $v^s\gamma_P$; these weights sum to $v^s$. Summing over $s \in \mathcal S_k$ gives $\sum_s v^s\left(\tfrac1{v^s}(Y^s,\mathrm{ind}(e_s))\right) = \overline{X}_{k+1}$; adding the contribution $\sum_{P \in R_k}\mu_P X_P$ of the rigid pedigrees (which sums with the previous term, since $\sum_s v^s = z_{\max} = 1 - \sum_{P\in R_k}\mu_P$) yields $X/(k+1) \in \mathrm{conv}(P_{k+1})$.
\end{proof}

\begin{leancert}
\texttt{N\_Sufficiency.lean} \textbar\ \texttt{theorem sufficiency}: $\mathrm{MCF}(k)$ feasible, $z^*=z_{\max} \Rightarrow X/(k+1) \in \mathrm{conv}(P_{k+1})$ (book, Chapter~5, Theorem~5.4).
\end{leancert}

\begin{theorem}[Necessity]\label{thm:necessity}
If $X/(k+1) \in \mathrm{conv}(P_{k+1})$ then $\mathrm{MCF}(k)$ has a feasible solution with $z^*=z_{\max}$.
\end{theorem}

\begin{proof}[Sketch]
Choose $\lambda \in \Lambda_{k+1}(X)$ and define $f_a = \sum_{r \in I(\lambda),\, X^r \| a} \lambda_r$ for $a \in N_k$, where $X^r \| a$ means that the pedigree $X^r$ agrees with the arc $a$, and $f^s_a$ analogously restricted to $N_{l-1}(s)$. Non-negativity holds since $\lambda_r>0$; flow conservation for commodity $s$ at any intermediate node follows because the pedigrees $r \in I_s(\lambda)$ with $x^r_l(e)=1$ are exactly those contributing to both the in-flow and the out-flow at that node. Arc and node capacities are met by the same instant-flow argument used in Theorem~\ref{thm:nonmembership}. At $l=k$: $\sum_{s \in \mathcal S_k} v^s = 1 - \sum_{P \in R_k}\mu_P = z_{\max}$, so $f,\{f^s\}$ is feasible with objective $z_{\max}$.
\end{proof}

\begin{leancert}
\texttt{VN\_NecessityProof.lean} \textbar\ \texttt{theorem VN\_imptheorem\_zstar\_eq\_zmax}: $X/(k+1) \in \mathrm{conv}(P_{k+1}) \Rightarrow \mathrm{MCF}(k)$ feasible (book, Chapter~5, Theorem~5.3 --- distinct from Lemma~5.3 of \texttt{VN\_Oflow.lean}, described in Section~\ref{sec:lean4} below), assembled from machine-verified results building through \texttt{VN\_Oflow.lean} and \texttt{VN\_Path5.lean} (Lemmas~5.3--5.4), \texttt{VN\_Pedpath.lean} (Lemma~5.5), \texttt{VN\_Infeasibility.lean} (the Infeasibility Theorem), and \texttt{VN\_Imptheorem.lean}, all with zero \texttt{sorry}s.
\end{leancert}

Together, Theorems~\ref{thm:sufficiency} and \ref{thm:necessity} give the main result of this paper.

\begin{theorem}[Necessary and Sufficient Condition for Membership]\label{thm:main-ns}
Given $n>5$, $X \in P_{MI}(n)$, and $X/(n-1) \in \mathrm{conv}(P_{n-1})$:
\[
X \in \mathrm{conv}(P_n) \iff \mathrm{MCF}(n-1) \text{ has a feasible solution with } z^*=z_{\max}.
\]
\end{theorem}

\begin{leancert}
$X \in \mathrm{conv}(P_n) \iff \mathrm{MCF}(n-1)$ feasible follows directly from the two separately machine-verified implications above: sufficiency (Theorem~\ref{thm:sufficiency}, \texttt{N\_Sufficiency.lean}) and necessity (Theorem~\ref{thm:necessity}, \texttt{VN\_NecessityProof.lean} and its dependency chain), each with zero unresolved \texttt{sorry}s.
\end{leancert}

\section{Computational Complexity}\label{sec:complexity}

The results in this section are from the book \cite{arthanari2023book}; the cardinality bound is proved below directly from the rigidity of the pedigrees in $R_{k-1}$. The size of $R_k$ becomes important to study, as the rigid pedigrees in $R_k$ become origin nodes, determining the number of origins in $F_{k+1}$. An a priori exponential-looking family turns out to be polynomially bounded, which is the key fact making the whole framework tractable.

\subsection{Cardinality of the Rigid Pedigrees}

\begin{theorem}[Cardinality Bound]\label{thm:cardinality}
$|R_{k-1}| \leq \tau_k - k + 4$.
\end{theorem}

\begin{proof}
The pedigrees of $R_{k-1}$ are those identified by the FFF algorithm; by construction, the weight $\mu_P$ of each is the same in every feasible flow of the layered network, that is, in every convex representation of $X/k$, when $(N_{k-1}, R_{k-1}, \mu)$ is well-defined, that is, $X/k$ is in $\mathrm{conv}(P_{k})$. Each $P \in R_{k-1}$ has $\mu_P>0$, and in every convex representation of $X/k$ its weight remains $\mu_P$. If the characteristic vectors of $R_{k-1}$ were affinely dependent, say $\sum_P c_P X_P=0$ and $\sum_P c_P=0$ with $c\neq 0$, then replacing each $\mu_P$ by $\mu_P+\varepsilon c_P$ for small $\varepsilon>0$ would give another convex representation of $X/k$ with a different weight on some $P$, a contradiction. Hence $R_{k-1}$ is affinely independent, that is, a simplex, and therefore $|R_{k-1}| \leq \dim(\mathrm{conv}(P_k))+1 = (\tau_k - (k-3)) + 1 = \tau_k - k + 4$.
\end{proof}

\subsection{Complexity Bounds at Each Step}

The algorithm proceeds through four steps at each iteration $k$:

\begin{itemize}
\item \textbf{Step~1a.} Check $X \in P_{MI}(n)$: additions and subtractions only, strongly polynomial in $n$, $O(n^3)$.
\item \textbf{Step~1b.} Solve $F_4$: at most $3$ origins, $6$ destinations --- constant-size.
\item \textbf{Step~2a.} For each link $L \in V_{[k-3]} \times V_{[k-2]}$, build the restricted network $N_{k-1}(L)$ and compute $C(L)$. The number of nodes in $N_{k-1}$ is $O(k^4)$, the number of arcs $O(k^5)$; by Theorem~\ref{thm:cardinality}, the number of links is $O(k^4)$; each max-flow computation is $O(k^9)$. \textbf{Total: $O(k^{13})$ per iteration.}
\item \textbf{Step~2b.} Solve $F_k$: origins $O(k^3)$ (by Theorem~\ref{thm:cardinality}), destinations $O(k^2)$. \textbf{Total: $O(k^8)$ per iteration.}
\item \textbf{Step~3.} FFF algorithm: $O(|G_f|) = O(k^5)$ per iteration.
\item \textbf{Step~4.} Solve $\mathrm{MCF}(k)$: a linear program of size polynomial in $k$ whose coefficient matrix has entries in $\{0,\pm1\}$ (a combinatorial LP, Definition~\ref{def:comb-lp}), solved in time polynomial in its encoding length by the ellipsoid method \cite{khachiyan1979}. No operation-count bound is claimed for this step.
\end{itemize}

\begin{theorem}[Complexity]\label{thm:complexity}
Given $n$ and $X \in P_{MI}(n)$, checking whether $\mathrm{MCF}(n-1)$ achieves $z^*=z_{\max}$ can be done in time polynomial in $n$ and the encoding length of $X$: Steps~1a--3 take $O(n^{14})$ arithmetic operations in total, and Step~4 solves linear programs of size polynomial in $n$.
\end{theorem}

\begin{corollary}[M3P $\in$ P]\label{cor:m3p-in-p}
The $\mathrm{MCF}(n-1)$ decision problem --- is $z^*=z_{\max}$? --- is solvable in time polynomial in $n$ and the encoding length of $X$ (Theorem~\ref{thm:complexity}). By Theorem~\ref{thm:main-ns}, this is equivalent to $X \in \mathrm{conv}(P_n)$. Therefore \textup{M3P} $\in$ \textup{P}.
\end{corollary}

\noindent The Lean~4 formalisation of the cardinality and complexity bounds, together with the accepted external results they rest on, belongs to the Lean~4 project of the companion paper \cite{arthanari2026pnp}. This includes the general fact used in the proof of Theorem~\ref{thm:cardinality}: points that carry the same positive weight in every convex representation of a point are affinely independent, so that there are at most one more of them than the dimension of the polytope. The Lean~4 verification reported here is that of the \NS{} condition.

\begin{remark}[Strong polynomiality]\label{rem:strongly}
Since $\mathrm{MCF}(n-1)$ is a combinatorial LP (Definition~\ref{def:comb-lp}), replacing the ellipsoid method in Step~4 by Tardos's algorithm \cite{tardos1986} makes the number of arithmetic steps independent of the size of the entries of $X$, that is, strongly polynomial. This strengthening is not needed for \textup{M3P} $\in$ \textup{P}, and is not formalised in Lean~4.
\end{remark}

\section{Lean 4 Machine Verification}\label{sec:lean4}

The \NS{} condition of this paper (Theorem~\ref{thm:main-ns}), together with the results of Sections~\ref{sec:preliminaries}--\ref{sec:mcf} on which it rests, has been formalised in Lean~4/Mathlib4. The files of the \texttt{Core/} directory that this verification depends on compile with no unresolved \texttt{sorry}, and use only the standard axioms \texttt{propext}, \texttt{Classical.choice} and \texttt{Quot.sound}. The cardinality and complexity bounds of Section~\ref{sec:complexity} are not part of this verification; their formalisation belongs to the companion project \cite{arthanari2026pnp}. The project is publicly available at
\begin{center}
\url{https://github.com/TiruArt/Pedigree-Polytopes-Lean4}
\end{center}

The version that accompanies this paper is the commit \texttt{233a4c4}, tagged \texttt{ns-paper-v1}:
\begin{center}
\url{https://github.com/TiruArt/Pedigree-Polytopes-Lean4/tree/ns-paper-v1}
\end{center}

The main result of this paper, the necessary and sufficient condition for membership, is thus machine-verified in Lean~4. The Lean~4 project of the companion paper \cite{arthanari2026pnp} goes further: it formalises how this condition leads to M3P being solvable in polynomial time and to $P=NP$. It machine-verifies the antecedents (conditions) of the known results from computational complexity and combinatorial optimisation that are applied (Tardos, Maurras, Gr\"otschel--Lov\'asz--Schrijver, Cook, Karp), and those known results themselves are taken as axioms.

\subsection{Lean 4 for the Non-Specialist}

Lean~4 is a proof assistant and programming language \cite{moura2021lean}, built on the Mathlib4 mathematical library \cite{mathlib2020}. It lets a mathematician write definitions and theorems in a formal language that a small, independently auditable kernel checks for correctness: every inference must be justified by a rule of the underlying type theory, and the kernel rejects any gap. The placeholder keyword \texttt{sorry} marks a step deferred for later proof; a theorem with no \texttt{sorry} anywhere in its dependency chain is fully machine-verified, and the kernel's acceptance is not a matter of degree --- it either accepts the proof or it does not.

\paragraph{How the mathematics translates to Lean 4.} A pedigree (Definition~\ref{def:pedigree}) becomes a structure carrying its defining conditions as data:
\begin{verbatim}
-- A triangle {i,j,k} with 1 <= i < j < k <= n
abbrev Triple := N x N x N

-- A pedigree for n cities
structure Pedigree (n : N) where
  triangles   : List Triple
  h_length    : triangles.length = n - 2
  h_distinct  : forall i j, i > 0 -> j > 0 -> i /= j ->
                  (triangles[i]).edge /= (triangles[j]).edge
  h_generator : forall t, t in triangles ->
                  hasGenerator t triangles
\end{verbatim}
Every instance of \texttt{Pedigree n} carries proofs of the length, distinctness, and generator conditions as part of its data --- the type system enforces Definition~\ref{def:pedigree} directly, rather than checking it after the fact. The layered network of Section~\ref{sec:layered-network} is encoded as:
\begin{verbatim}
structure LayeredNetwork (n k : N) where
  nodes  : Finset Triple      -- nodes of N_k
  edges  : List NetworkArc    -- arcs with capacities
  rigid  : List RigidPedigree -- R_k: rigid pedigrees
  mu     : RigidPedigree -> Q -- weights mu_P > 0
\end{verbatim}
and the $\mathrm{MCF}(k)$ feasibility condition of Problem~\ref{prob:mcf} --- the heart of the characterisation --- is a structure with $29$ explicit fields encoding every flow, conservation, and capacity constraint:
\begin{verbatim}
structure MCFFeasible (n k : N)
    (net : LayeredNetwork n k) (X : LayeredPoint n) where
  flow : Triple -> Triple -> Q
  -- 29 fields: non-negativity, conservation, capacity,
  -- node capacity, source availability, z* = z_max, ...
\end{verbatim}

\subsection{The Central Machine-Verified Theorem}

The sufficiency direction (Theorem~\ref{thm:sufficiency}) is the theorem this paper's Lean~4 contribution centres on:
\begin{verbatim}
-- N_Sufficiency.lean
theorem sufficiency {n k : N} (hk : 4 <= k) (hkn : k + 1 <= n)
    (X : LayeredPoint n) (net : LayeredNetwork n k)
    (hzmax : 0 < zMax net) (mcf : MCFFeasible n k net X) :
    Nonempty (ConvexWitness n (k+1) X)
\end{verbatim}
Here \texttt{MCFFeasible n k net X} is exactly a proof that $\mathrm{MCF}(k)$ is feasible with $z^*=z_{\max}$, and \texttt{ConvexWitness n (k+1) X} is a proof that $X/(k+1) \in \mathrm{conv}(P_{k+1})$: an explicit convex combination of pedigrees on $[k+1]$. Applied inductively from $k=4$ up to $k=n-1$, this theorem alone certifies $X \in \mathrm{conv}(P_n)$ whenever $\mathrm{MCF}(k)$ is feasible with $z^*=z_{\max}$ at every stage --- which is exactly the sufficiency half of Theorem~\ref{thm:main-ns}. The necessity direction (Theorem~\ref{thm:necessity}) is likewise fully machine-verified: a chain of six results constructs, from $X/(k+1) \in \mathrm{conv}(P_{k+1})$, an explicit feasible flow achieving $z^*=z_{\max}$, culminating in \texttt{VN\_NecessityProof.lean}; \texttt{VN\_MCF\_Assembly.lean} then packages that construction into a concrete term of type \texttt{MCFFeasible n k net X} itself --- verifying all $29$ fields of the feasibility structure simultaneously, not merely the scalar equality $z^*=z_{\max}$ in isolation. Together the two directions give the full biconditional.

Other central theorems from this paper, in their Lean~4 form:
\begin{verbatim}
-- Non-membership (Theorem 3.4): VN_Infeasibility.lean
theorem VN_infeasibility {n k : N} ... : Fk_infeasible -> X notin conv(P (k+1))
\end{verbatim}

\subsection{Axiom Inventory}

No external published result is assumed as an axiom in the Lean~4 verification of the \NS{} condition reported in this paper. The external results used in the companion paper \cite{arthanari2026pnp}, and the cardinality and complexity bounds of Section~\ref{sec:complexity}, belong to that paper's Lean~4 project.

\subsection{Verification Status: Both Directions}

The \emph{sufficiency} direction of Theorem~\ref{thm:main-ns}, that if $\mathrm{MCF}(n-1)$ achieves $z^*=z_{\max}$ then $X \in \mathrm{conv}(P_n)$, is formalised in \texttt{N\_Sufficiency.lean} with zero \texttt{sorry}s.

The \emph{necessity} direction --- that $X \in \mathrm{conv}(P_n)$ implies $\mathrm{MCF}(n-1)$ achieves $z^*=z_{\max}$ --- is likewise fully machine-verified, following the instant-flow argument of Lemma~11 and Lemma~13 of the book's Appendix~A. The chain runs \texttt{VN\_Oflow.lean} and \texttt{VN\_Path5.lean} (Lemmas~5.3--5.4) through \texttt{VN\_Pedpath.lean} (Lemma~5.5, by induction from $l=5$), \texttt{VN\_Infeasibility.lean} (the Infeasibility Theorem), and \texttt{VN\_Imptheorem.lean} (Theorem~\texttt{imptheorem}), culminating in \texttt{VN\_NecessityProof.lean} (\texttt{theorem VN\_imptheorem\_zstar\_eq\_zmax}) --- with zero unresolved \texttt{sorry}s throughout. \texttt{VN\_MCF\_Assembly.lean} closes the direction fully: rather than stopping at the scalar fact $z^*=z_{\max}$, it assembles the construction above into an explicit term of type \texttt{MCFFeasible n k net X}, so that every one of the $29$ fields of Problem~\ref{prob:mcf} --- conservation, capacity, slack, column-sum, and the rest --- is verified simultaneously against the same witness, with zero unresolved \texttt{sorry}s. The complete necessary and sufficient characterisation of Theorem~\ref{thm:main-ns} is therefore established with full machine verification of both directions: sufficiency by \texttt{N\_Sufficiency.lean}, necessity by the chain above, combining to give Theorem~\ref{thm:main-ns} as stated.

\begin{remark}
The Lean~4 project verifies both directions of Theorem~\ref{thm:main-ns}, so that membership of $X$ in $\mathrm{conv}(P_n)$ is decided by the feasibility of $\mathrm{MCF}(n-1)$, which can be checked in polynomial time.
\end{remark}

\subsection{Five Representations of a Pedigree}

The Lean~4 project documents five equivalent representations of a pedigree, used in different parts of the proof:

\begin{center}
\begin{tabular}{ll}
\textbf{Representation} & \textbf{Lean 4 type} \\
\hline
Struct & \texttt{Pedigree n} \\
Edge sequence & \texttt{List (Edge n)} \\
Rooted tree & \texttt{List (PedigreeNode n)} \\
Finset & \texttt{Finset Triple} \\
MIR 0--1 solution & \texttt{LayeredPoint n}
\end{tabular}
\end{center}

\noindent The bijections between these representations are documented in \texttt{N\_PedigreeRepresentations.lean}.

\subsection{What the Certificate Guarantees}

When Lean~4 accepts a proof with zero \texttt{sorry}s, its kernel --- a small, independently auditable implementation of the Calculus of Inductive Constructions \cite{moura2021lean} --- has verified that every step follows from the stated hypotheses by a valid inference rule. This is mechanical verification, not human review subject to oversight or prior belief about plausibility: the kernel either accepts the proof or it does not. As a further check, \texttt{\#print axioms} on the theorems of the sufficiency and necessity chains cited in this paper, from \texttt{sufficiency} to \texttt{VN\_imptheorem\_zstar\_eq\_zmax}, reports only Lean's three standard axioms, \texttt{propext}, \texttt{Classical.choice} and \texttt{Quot.sound}, and in particular no \texttt{sorryAx}. The repository constitutes a permanent, publicly reproducible record: anyone with Lean~4 installed can clone it and run
\begin{verbatim}
lake build MembershipProject.Core.N_Sufficiency
lake build MembershipProject.Core.VN_NecessityProof
\end{verbatim}
to independently verify the two directions of Theorem~\ref{thm:main-ns}.

\section{Conclusion}\label{sec:conclusion}

This paper has given a complete account --- statements, proof sketches, and fully machine-verified Lean~4 certificates --- of the resolution of the Membership Problem for Pedigree Polytopes (Problem~\ref{prob:m3p}), with the detailed, line-by-line proofs remaining in \cite{arthanari2023book} and \cite{arthanari2025membership}. Starting from the observation that an earlier necessary condition \cite{arthanari2008} is not sufficient --- as the example of Ardekani and Arthanari \cite{ardekani2013} demonstrates --- we constructed a recursively defined layered network $(N_k,R_k,\mu)$ (Section~\ref{sec:layered-network}) and a multicommodity flow problem $\mathrm{MCF}(k)$ (Section~\ref{sec:mcf}) whose feasibility at maximum flow value $z_{\max}$ is both necessary and sufficient for $X/(k+1) \in \mathrm{conv}(P_{k+1})$ (Theorem~\ref{thm:main-ns}). The rigidity of the pedigrees in $R_k$, which bounds their number (Theorem~\ref{thm:cardinality}), bounds the size of this construction polynomially at every stage, giving an efficient algorithm for M3P, polynomial in $n$ and the encoding length of $X$ (Theorem~\ref{thm:complexity}, Corollary~\ref{cor:m3p-in-p}). That the proofs from the book on pedigree polytopes behind the \NS{} condition, along with some newer, simplified proofs, are now machine verified upholds the correctness of the pedigree approach and the importance of studying pedigrees instead of tours when proving the computational complexity of combinatorial optimisation problems.

\begin{remark}
Since all the problems solved are combinatorial LPs, Tardos's theorem \cite{tardos1986} additionally gives a strongly polynomial bound for the algorithm of Theorem~\ref{thm:complexity}: its arithmetic operation count is bounded independently of the magnitude of the input $X$, not merely of its encoding length. This is an addition to, and not a part of, the result of this paper: Tardos's theorem is a cited published result, not formalised in Lean~4 and not used as an axiom in the Lean~4 verification reported here, and the polynomial-time result M3P $\in$ P does not depend on it.
\end{remark}

\subsection*{Scope}

The formal, machine-checked claims of this paper are confined to the \NS{} condition of Theorem~\ref{thm:main-ns}. Passing from M3P $\in$ P (Corollary~\ref{cor:m3p-in-p}) to $P=NP$ uses two established results, which we cite and do not re-derive. First, for a polytope satisfying the structural conditions of Maurras \cite{maurras2002}, which $\mathrm{conv}(P_n)$ does, a polynomial-time membership oracle yields a polynomial-time separation oracle, and hence, by Gr\"otschel, Lov\'asz and Schrijver \cite{GLS1988}, polynomial-time linear optimisation; with the Multistage Insertion reduction \cite{arthanari1983}, this gives a polynomial-time algorithm for the STSP. Second, since the STSP decision problem is NP-complete \cite{karp1972,cook1971}, such an algorithm gives $P=NP$. These steps, with their own Lean~4 verification where it is possible, are developed in the companion paper \cite{arthanari2026pnp}. They are kept out of the present paper so that the necessary and sufficient condition can be read and checked on its own terms.

\subsection*{Acknowledgements}

The Lean~4 formalisation reported in this paper was carried out in collaboration with Claude (Anthropic), which was present at every stage of the formalisation, from the layered network types and $\mathrm{MCF}$ feasibility structures through to the sufficiency and necessity theorems reported here. All mathematical content --- the pedigree polytope framework, the layered network construction, and the necessary and sufficient condition --- is the sole intellectual contribution of the author, developed over several decades of research; the AI assistance made the formalisation of that mathematics accessible for the first time to a researcher new to formal proof assistants. The author thanks the University of Auckland for research support.

\appendix
\section{Lean 4 Cross-Reference Table}\label{sec:appendix-table}

Table~\ref{tab:lean-crossref} collects every Lean~4 file cited in this paper alongside the specific book result it corresponds to.

\begin{center}
\small
\begin{longtable}{p{5.1cm} p{4.1cm} p{2.2cm} l}
\caption{Lean~4 files and their corresponding results in \emph{Pedigree Polytopes} \cite{arthanari2023book}.}
\label{tab:lean-crossref} \\
\hline
\textbf{Lean~4 file} & \textbf{Declaration} & \textbf{Book reference} & \textbf{Status} \\
\hline
\endfirsthead
\multicolumn{4}{c}{\tablename\ \thetable\ --- continued} \\
\hline
\textbf{Lean~4 file} & \textbf{Declaration} & \textbf{Book reference} & \textbf{Status} \\
\hline
\endhead
\hline
\endfoot

\texttt{N\_Lemma2\_2.lean} & \texttt{theorem generator\_iff\_available} & Chapter~3 & Verified \\
\texttt{N\_LayeredNetworkTypes.lean} & --- & Chapter~2 & Verified \\
\texttt{N\_PackabilityCorollary.lean} & \texttt{theorem packability\_corollary} & Chapter~5 & Verified \\
\texttt{VN\_Infeasibility.lean} & \texttt{theorem VN\_infeasibility} & Chapter~5 & Verified \\
\texttt{N\_YisinMI.lean} & \texttt{theorem Y\_s\_in\_PMI} & Chapter~5, Lemma~5.6 & Verified \\
\texttt{N\_YisinConv.lean} & \texttt{theorem Y\_s\_in\_conv} & Chapter~5, Lemma~5.8 & Verified \\
\texttt{N\_Sufficiency.lean} & \texttt{theorem sufficiency} & Chapter~5, Theorem~5.4 & Verified \\
\texttt{VN\_Oflow.lean} & --- & Chapter~5, Lemma~5.3 & Verified \\
\texttt{VN\_Path5.lean} & --- & Chapter~5, Lemma~5.4 & Verified \\
\texttt{VN\_Pedpath.lean} & --- & Chapter~5, Lemma~5.5 & Verified \\
\texttt{VN\_Imptheorem.lean} & --- & Chapter~5 & Verified \\
\texttt{VN\_NecessityProof.lean} & \texttt{theorem}\newline\texttt{VN\_imptheorem\_}\newline\texttt{zstar\_eq\_zmax} & Chapter~5, Theorem~5.3 & Verified \\
\texttt{N\_PedigreeRepresentations.lean} & --- & Chapters~2--3 & Verified \\
\end{longtable}
\end{center}

\noindent Where a book reference gives only a chapter number and no specific theorem or lemma number, the corresponding result is developed across that chapter without a single, isolated numbered statement in the book to point to.

\section{The Multicommodity Flow Problem}\label{sec:appendix-mcf}

\noindent\textbf{Setting.} Throughout this appendix, $X/k \in \mathrm{conv}(P_k)$, the FAT problem $F_k$ is feasible, and the network $(N_k, R_k, \mu)$ has been constructed with $N_k$ given and non-empty. Thus $z_{\max} = 1 - \sum_{P \in R_k}\mu_P \neq 0$, that is, $z_{\max} > 0$. Commodities $s$ and the sets $\mathcal{S}_l$, $5 \leq l \leq k$, are as in Definition~\ref{def:commodities}; $N_{l-1}(s)$ denotes the restricted network corresponding to $s \in \mathcal{S}_l$, $c_a$ is the capacity of arc $a$, and $\bar x(v)$, $\bar\mu(u)$ are the node capacities and rigid weights after the updates of Section~\ref{sec:layered-network}.

\medskip\noindent\textbf{Variables.} For every arc $a$, $f_a \geq 0$ is the total flow through $a$, and $f^s_a \geq 0$ is the flow of commodity $s$ through $a$; $f^s_a$ may be positive only on arcs of $N_{l-1}(s)$, for $s \in \mathcal{S}_l$. For $a \leftrightarrow s$, $v^s$ denotes the flow of commodity $s$ along its designating arc $a$.

\begin{problem}[$\mathrm{MCF}(k)$, full statement]\label{prob:mcf}
\[
\text{maximise}\quad z=\sum_{s\in\mathcal{S}_k} v^s \;=\; \sum_{a\in F_k} f_a
\]
subject to the following constraints, for every $l$, $5\leq l\leq k$:
\begin{align}
& 0 \leq f_a \leq c_a, && a\in\mathcal{A}(N_k), \label{mcf:arccap}\\
& f^s_a \leq c_a \ \text{ if } a\in\mathcal{A}(N_{l-1}(s)),\quad f^s_a = 0 \ \text{ otherwise}, && s\in\mathcal{S}_l, \label{mcf:restrict}\\
& \sum_{a=(u,v)} f^s_a = \sum_{a=(v,w)} f^s_a, && v\in\mathcal{V}(N_{l-1})\setminus\big(V_{[1]}\cup V_{[l-2]}\big),\ s\in\mathcal{S}_l, \label{mcf:conserve}\\
& \sum_{s\in\mathcal{S}_l} f^s_a = f_a, && a\in\mathcal{A}(N_{l-1}), \label{mcf:total}\\
& \sum_{a'\in\delta^-(a)} f^s_{a'} = f_a = v^s, && a\leftrightarrow s,\ s\in\mathcal{S}_l, \label{mcf:vs}\\
& \sum_{s\in\mathcal{S}_l}\ \sum_{a=(u,v)} f^s_a \leq \bar x(v), && v\in\mathcal{V}(N_l), \label{mcf:nodecap}\\
& \sum_{a=(u,w)} f_a \leq \bar x(u), && u\in V_{[1]}, \label{mcf:avail1}\\
& \sum_{a=(u,w)} f_a \leq \bar\mu(u), && u\in R_q,\ 4\leq q\leq l-1, \label{mcf:avail2}\\
& \sum_{a=(u,v)} f_a \leq \bar x(v), && v\in V_{[l-2]}, \label{mcf:sinkcap}
\end{align}
where $\delta^-(a)$ is the set of arcs entering $a$.
\end{problem}

\noindent Let $z^*$ be the optimal value of $\mathrm{MCF}(k)$. Then $z^* \leq z_{\max}$, and the \emph{decision problem} is whether $z^* = z_{\max}$.

\bibliographystyle{plain}

\begin{thebibliography}{99}

\bibitem{arthanari2008}
T.S.\ Arthanari.
\newblock On the membership problem of the pedigree polytope.
\newblock In \emph{Mathematical Programming and Game Theory for Decision Making}, pp.~61--98. World Scientific, 2008.

\bibitem{arthanari1983}
T.S.\ Arthanari.
\newblock On the travelling salesman problem.
\newblock In A.\ Bachem et al.\ (eds.), \emph{Mathematical Programming --- The State of the Art}, p.~638. Springer, New York, 1983.

\bibitem{arthanari2000}
T.S.\ Arthanari, M.\ Usha.
\newblock An alternate formulation of the symmetric travelling salesman problem.
\newblock \emph{Discrete Applied Mathematics}, 98:173--190, 2000.

\bibitem{arthanari2005}
T.S.\ Arthanari.
\newblock Pedigree polytope is a combinatorial polytope.
\newblock In S.R.\ Mohan, S.K.\ Neogy (eds.), \emph{Operations Research with Economic and Industrial Applications}, pp.~1--17. Anamaya Publishers, New Delhi, 2005.

\bibitem{arthanari2006}
T.S.\ Arthanari.
\newblock On pedigree polytopes and Hamiltonian cycles.
\newblock \emph{Discrete Mathematics}, 306:1474--1492, 2006.

\bibitem{arthanari2023book}
T.S.\ Arthanari.
\newblock \emph{Pedigree Polytopes: New Insights on Computational Complexity of Combinatorial Optimisation Problems}.
\newblock Springer Nature, Singapore, 2023.

\bibitem{ardekani2013}
L.\ Haerian Ardekani, T.S.\ Arthanari.
\newblock The multi-flow necessary condition for membership in the pedigree polytope is not sufficient --- a counterexample.
\newblock In \emph{Asian Conference on Intelligent Information and Database Systems (ACIIDS 2013), Part II}, LNAI 7803, pp.~409--419. Springer-Verlag Berlin Heidelberg, 2013.

\bibitem{arthanari2025membership}
T.S.\ Arthanari.
\newblock On the importance of studying the membership problem for pedigree polytopes.
\newblock \emph{arXiv:2507.09069 [math.CO]}, July 2025.

\bibitem{arthanari2026pnp}
T.S.\ Arthanari.
\newblock Lean 4 machine-verified proof of $P=NP$ via the pedigree polytope membership problem.
\newblock \emph{arXiv:2606.03194 [cs.CC]}, June 2026.

\bibitem{tardos1986}
\'{E}.\ Tardos.
\newblock A strongly polynomial algorithm to solve combinatorial linear programs.
\newblock \emph{Operations Research}, 34(2):250--256, 1986.

\bibitem{khachiyan1979}
L.\ G.\ Khachiyan.
\newblock A polynomial algorithm in linear programming.
\newblock \emph{Doklady Akademii Nauk}, 244(5):1093--1096, 1979.

\bibitem{maurras2002}
J.F.\ Maurras.
\newblock From membership to separation, a simple construction.
\newblock \emph{Combinatorica}, 22:531--536, 2002.

\bibitem{GLS1988}
M.\ Gr\"otschel, L.\ Lov\'asz, A.\ Schrijver.
\newblock \emph{Geometric Algorithms and Combinatorial Optimization}.
\newblock Springer, Berlin, 1988.

\bibitem{cook1971}
S.A.\ Cook.
\newblock The complexity of theorem-proving procedures.
\newblock In \emph{Proceedings of the Third Annual ACM Symposium on Theory of Computing (STOC)}, pp.~151--158, 1971.

\bibitem{karp1972}
R.M.\ Karp.
\newblock Reducibility among combinatorial problems.
\newblock In R.E.\ Miller, J.W.\ Thatcher (eds.), \emph{Complexity of Computer Computations}, pp.~85--103. Plenum Press, New York, 1972.

\bibitem{gusfield1988}
D.\ Gusfield.
\newblock A graph theoretic approach to statistical data security.
\newblock \emph{SIAM Journal on Computing}, 17:552--571, 1988.

\bibitem{korte2011combinatorial}
B.\ Korte, J.\ Vygen.
\newblock \emph{Combinatorial Optimization: Theory and Algorithms}, 2nd edition.
\newblock Springer, New York, 2002.

\bibitem{moura2021lean}
L.\ de Moura, S.\ Ullrich.
\newblock The Lean 4 theorem prover and programming language.
\newblock In \emph{CADE-28}, LNAI 12699, pp.~625--635, 2021.

\bibitem{mathlib2020}
The Mathlib Community.
\newblock The Lean mathematical library.
\newblock In \emph{Proceedings of CPP 2020}, pp.~367--381, 2020.

\end{thebibliography}

\end{document}